\documentclass[leqno,12pt]{amsart} %leqno is the option to put formula numbers on the left side
\usepackage[top=30truemm,bottom=30truemm,left=25truemm,right=25truemm]{geometry}
\usepackage{amssymb}
\usepackage{amsmath}
\usepackage{amsthm}
\usepackage{amscd}
\usepackage{mathrsfs}
\usepackage{graphicx}
\usepackage[dvips]{color}
\usepackage[all]{xy}
\usepackage{url}
\usepackage[varg]{txfonts}
  \makeatletter

\@addtoreset{equation}{section}
\@namedef{subjclassname@2020}{\textup{2020} Mathematics Subject Classification}
\makeatother
\theoremstyle{plain} %text of this environment is typesetted in italics
\newtheorem{theorem}{\indent\bf Theorem}[section]

\theoremstyle{definition} %text of this environment is typesetted in roman letters
\newtheorem{definition}[theorem]{\indent\bf Definition}
\newtheorem{remark}[theorem]{\indent\bf Remark}

\newcommand{\dbar}{\overline{\partial}}
\newcommand{\ai}{\sqrt{-1}}

\newcommand{\R}{\mathbb{R}}

\newcommand{\C}{\mathbb{C}}
\newcommand{\N}{\mathbb{N}}

\newcommand{\dl}{\mathrm{d}\lambda}

\newcommand{\Rea}{\mathrm{Re}}
\newcommand{\Ima}{\mathrm{Im}}

\begin{document}
\pagestyle{plain}
\thispagestyle{plain}

\title[]
{Optimal $L^2$-extensions on tube domains and a simple proof of Pr\'ekopa's theorem}

\author[T. INAYAMA]{Takahiro INAYAMA}
\address{Department of Mathematics\\
Faculty of Science and Technology\\
Tokyo University of Science\\
2641 Yamazaki, Noda\\
Chiba, 278-8510\\
Japan
}
\email{inayama\_takahiro@ma.noda.tus.ac.jp}
\email{inayama570@gmail.com}
%\date{\today}

%%%%%%%%%%%%%%%%%%%%%%%%%%%%%%%%%%%%%%%%%%%%%%%%%%%%%%
\begin{abstract}
We prove the optimal $L^2$-extension theorem of Ohsawa-Takegoshi type on a tube domain. 
As an application, we give a simple proof of Pr\'ekopa's theorem. 
\end{abstract}

%%%%%%%%%%%%%%%%%%%%%%%%%%%%%%%%%%%%%%%%%%%%%%%%%%%%%%

\maketitle
\setcounter{tocdepth}{2}
%\tableofcontents
%%%%%%%%%%%%%%%%%%%%%%%%%%%%%%%%%%%%%%%%%%%%%%%%%%%%%%

\section{Introduction}
Pr\'ekopa's theorem \cite{Pre73}, which can be seen as a generalization of the Brunn-Minkowski theorem, plays an important role in convex geometry. 
The theorem asserts that if $\varphi: \R_{t}\times \R^{n}_{x}\to \R $ is a convex function, the function $\Phi:\R\to \R$ defined by 
$$
e^{-\Phi(t)}:= \int_{\R^n}e^{-\varphi(t,x)} \dl(x)
$$
is also convex. 

Replacing $\R$ by $\C$ and convex functions by plurisubharmonic functions, we can consider a version of Pr\'ekopa's theorem in the complex setting. 
Unfortunately, it is known that this complex Pr\'ekopa problem does not hold in general (see \cite{Kis78}). 
However, Berndtsson \cite[Theorem 1.3, 2]{Ber98} proved that if a plurisubharmonic function $\varphi: D\times (V+\ai \R^n)\subset \C_\tau\times \C_z^n\to \R\cup \{ -\infty\}$ is independent of $\Ima(z)$, 
the function $\Phi$ on $D$ defined by 
$$
e^{-\Phi(\tau)}:= \int_{V}e^{-\varphi(\tau, \Rea (z))}\dl (\Rea (z))
$$
is plurisubharmonic as well, where $V\subset \R^n$ is a convex domain and $V+\ai \R^n:=\{ z=x+\ai y\in \C^n \mid x\in V\}$ is a tube domain. 
The above assumption of $\varphi$ is appropriate in the following sense. 
If $\varphi$ is a convex function on $V$, the associated function $\widehat{\varphi}(z):= \varphi(x)$ is plurisubharmonic on $V+\ai \R^n$. 
Conversely, if $\widehat{\varphi}$ is  plurisubharmonic on $V+\ai \R^n$ and independent of $\Ima (z)$, the well-defined function $\varphi(x):=\widehat{\varphi}(x+\ai \R^n)$ is convex on $V$. 
This simple observation allows us to study the convexity of functions via complex analytic methods. 
For the Pr\'ekopa theorem and the complex Pr\'ekopa theorem, one main tool to prove them is the $L^2$-\textit{estimate} of $d$ or $\dbar$ equation (see e.g. \cite{BL76}, \cite{Ber98}). 

In this article, we give a proof of Pr\'ekopa's theorem by using $L^2$-\textit{extension} theorems without any regularity assumption or direct computation of curvature. 
In order to give the proof, we prove the following optimal $L^2$-extension theorem. 

\begin{theorem}\label{thm:optimaltube}
	Let $D\subset \C_\tau$ be a domain, $V$ be a bounded convex domain in $\R^n_{x}$ and $V_{x}+\ai \R^n_{y} \subset \C^n_{z}$ be a tube domain. 
	Assume that $\varphi(\tau, z)$ is a plurisubharmonic function on $D\times (V+\ai \R^n)$, which is independent of $y=\Ima(z)$. 
	Then, for any point $a\in D$ and any $r>0$ such that $\int_V e^{-\varphi(a, x)}\dl(x)<+\infty$ and $\Delta(a;r)=\{ |\tau-a|<r\}\subset D$, there exists a holomorphic function $f$ on $\Delta(a;r)$ satisfying $f(a)=1$ and 
	$$
	\int_{\Delta(a;r)\times V}|f(\tau)|^2e^{-\varphi(\tau,x)}\dl(\tau,x)\leq \pi r^2\int_V e^{-\varphi(a,x)}\dl (x). 
	$$
\end{theorem}

This is a version of the optimal $L^2$-extension theorem due to \cite{Blo13},\cite{GZ15}, initially proved by Ohsawa and Takegoshi \cite{OT87} for some constant, not necessarily optimal. 
%The above constant $\pi r^2$ is known to be the optimal by \cite{Blo13}, \cite{GZ15}. 

The proof of Theorem \ref{thm:optimaltube} is a little bit \textit{complex}.
On the other hand, if we regard the optimal $L^2$-extension theorem above as a fact, 
we can give a quite simple proof of Pr\'ekopa's theorem. 
A key notion is \textit{the minimal extension property} or \textit{the optimal $L^2$-extension property}, which is introduced in \cite{HPS18} or \cite{DNW19}, \cite{DNWZ20}, respectively.
 
 \vskip10mm
 {\bf Acknowledgment. }
 The author would like to thank Bo Berndtsson for reading and commenting on a draft version.
 He is also grateful to the anonymous referee for careful reading and pointing out a gap in the proof of the main theorem. 
 %He is supported by JSPS KAKENHI Grant Number 18J22119. 

\section{Optimal $L^2$-extensions and minimal extension property}

In this article, we let $\lambda_n$ denote the standard Lebesgue measure on $\R^n$ and omit $n$.
%In this section, we introduce \textit{the optimal $L^2$-extension theorem} and the notion of {\it the minimal extension property} and {\it the optimal $L^2$-extension property}.
First, we introduce the optimal $L^2$-extension theorem in the following form. 

\begin{theorem}[\cite{Blo13}, \cite{GZ15}]\label{thm:blockiguanzhou}
	Let $D$ be a bounded pseudoconvex domain with $D\subset \C^{n-1}\times \{ |z_n|<r\}$ for $r>0$. 
	We also let $\varphi$ be a plurisubharmonic function on $D$ and $H:=\Omega \cap \{z_n=0 \}$. 
	Then for any holomorphic function $f$ on $H$ with $\int_H |f(z')|^2e^{-\varphi(z',0)}\dl(z')<+\infty$, there exists a holomorphic function $F$ on $D$ satisfying $F|_H=f$ and 
	$$
	\frac{1}{\pi r^2}\int_{D}|F(z',z_n)|^2e^{-\varphi(z', z_n)}\dl(z',z_n)\leq \int_H |f(z')|^2e^{-\varphi(z',0)}\dl(z'),
	$$
	where $(z')=(z_1, \cdots, z_{n-1})\in \C^{n-1}$.
\end{theorem}

Then we introduce the notion of {\it the minimal extension property} and {\it the optimal $L^2$-extension property}
 (hereafter, we will use the former term). 

\begin{definition}[minimal extension property \cite{HPS18}, the optimal $L^2$-extension property \cite{DNW19}, \cite{DNWZ20}]
	Let $\varphi:D\to \R\cup \{ -\infty\}$ be an upper semi-continuous function on a domain $D\subset \C$. We say that $\varphi$ satisfies {\it minimal extension property} if for any $a\in D$ with $\varphi(a)\neq -\infty$ and for any $r>0$ satisfying $\Delta(a; r)\subset D$,  there exists a holomorphic function on $\Delta(a;r)$ such that $f(a)=1$ and 
	$$
	\frac{1}{\pi r^2}\int_{\Delta(a;r)}|f|^2 e^{-\varphi}\dl \leq e^{-\varphi(a)}. 
	$$
\end{definition}

Note that the minimal extension property can be defined for an $n$-dimensional domain. 
In this paper, we only consider the case $n=1$. 
If $\varphi$ is plurisubharmonic, due to Theorem \ref{thm:blockiguanzhou}, 
$\varphi$ satisfies the above minimal extension property.
As a converse, it is known that the following result holds. 

\begin{theorem}$($\cite[Theorem 1.4]{DNW19}, cf. \cite{GZ15}, \cite{HPS18},  \cite{DNWZ20}$)$.\label{thm:mep} 
	Keep the notation above. If an upper semi-continuous function $\varphi$ satisfies the minimal extension property, $\varphi$ is plurisubharmonic. 
\end{theorem}

This type of idea was initially observed by Guan and Zhou in \cite{GZ15}. 
For the sake of completeness, we give the proof. 

\begin{proof}
	It is enough to show that $\varphi$ satisfies the mean value inequality at any point $a\in D$ with $\varphi(a)>-\infty$. 
	Take any $r>0$ satisfying $\Delta(a; r)\subset D$. 
	Thanks to the assumption, we can take a holomorphic function $f$ on $\Delta(a; r)$ satisfying $f(a)=1$ and 
	$$
	\frac{1}{\pi r^2}\int_{\Delta(a; r)}|f|^2e^{-\varphi} \dl \leq e^{-\varphi(a)}. 
	$$
	Taking logarithms and using Jensen's inequality, we have 
	\begin{align*}
	-\varphi(a) &\geq \log \left( \int_{\Delta(a; r)} |f|^2e^{-\varphi}\frac{\dl}{\pi r^2} \right)\\
	&\geq \frac{1}{\pi r^2}\int_{\Delta(a; r)}\log |f|^2\dl - \frac{1}{\pi r^2}\int_{\Delta(a; r)} \varphi \dl.
	\end{align*}
	Since $\log |f|^2$ is plurisubharmonic and $f(a)=1$, we obtain 
	$$
	\frac{1}{\pi r^2}\int_{\Delta(a; r)} \varphi \dl \geq \varphi(a). 
	$$
\end{proof}

\section{Optimal $L^2$-extension theorems on tube domains}\label{sec:optimall2}

In this section, we prove Theorem \ref{thm:optimaltube}. 
Ohsawa-Takegoshi type $L^2$-extension theorems usually require the boundedness of domains. 
To extend holomorphic functions on unbounded domains such as tube domains, we take a functional analytic approach. 
The proof is inspired by the method in \cite{Ber98}. 
%Here we present the proof in as much detail as possible. 
Throughout the proof, we simply write $y$ instead of some $y_i$ (for example, $\frac{\partial }{\partial y}$). 
We also say that a function $f$ is holomorphic on a non-open set $K$ if $f$ is holomorphic on some open neighborhood $U$ of $K\subset U$. 

\begin{proof}[\indent Proof of Theorem \ref{thm:optimaltube}.]
	The proof is divided into three steps. 
	
	(Step 1) Construct holomorphic functions on each bounded domain.
	
	Let $B_R\subset \R^n$ denote $B_R:=\{ y=(y_1,\cdots, y_n)\in \R^n \mid |y|^2=|y_1|^2+\cdots +|y_n|^2<R^2\}$ for $R>0$. 
	Consider a constant function $1$ on $\{a\}\times (V+\ai B_R)$. 
	Then, due to Theorem \ref{thm:blockiguanzhou}, we get a holomorphic function $f_R$ on $\Delta(a;r)\times (V+\ai B_R)$ satisfying $f_R|_{\{a\}\times (V+\ai B_R)}\equiv 1$ and 
	\begin{align}
	\int_{\Delta(a;r)\times (V+\ai B_R)}|f_R|^2e^{-\varphi(\tau, x, y)}\dl (\tau, x, y) &\leq \pi r^2 \int_{(V+\ai B_R)} e^{-\varphi(a, x, y)}\dl (x, y)\\
	&\leq \pi r^2 (\sigma_n R^n)\int_V e^{-\varphi(a, x)}\dl (x)\label{eq:roptimal}
	\end{align}
	for each $R>0$ since $\varphi$ is independent of $y$. Here $\sigma_n$ is the volume of the unit ball in $\R^n$. 
	Roughly speaking, we would like to consider the limit $\lim_{R\to +\infty}f_R/ \sqrt{\sigma_nR^n}$. To do this procedure precisely, we take a convolution of $f_R$ with bump functions. 
	
	(Step 2) Take a convolution and estimate $L^2$ norms. 
	
	Define a bump function $\chi_R(y)$ on $\R^n_y$ as follows: $0\leq \chi_R\leq 1$, $\chi_R$ is smooth and has compact support in $\{ |y|<R-\sqrt{R} \}$, $\chi_R|_{\{ |y|<R-2\sqrt{R}\}}\equiv 1$ and $|\nabla \chi_R|\leq C/\sqrt{R}$ for some positive constant $C>0$. 
	We also let $a_R:=\int_{\R^n}\chi_R(w)\dl(w)$. 
	Take a convolution of $f_R$ with $\chi_R/a_R$
	$$
	\widetilde{f}_R(\tau,x,y):= \frac{1}{a_R} \int_{\R^n}f_R(\tau,x,y-w)\chi_R(w)\dl(w)=\frac{1}{a_R} \int_{\R^n}\chi_R(y-w)f_R(\tau, x, w)\dl(w).
	$$
	Here we regard $f_R\equiv 0$ on $\Delta(a;r)\times (V+\ai (\R^n\setminus \overline{B_R}))$ and take the convolution on $\R^n$. 

	Note that $f_R(\tau,x,w)\equiv 0$ and $\chi_R(w)\equiv 0$ if $w\in \R^n\setminus \overline{B_R}$. 
	Then we have that 
	%$$
	%\widetilde{f}_R(a,x,y)=\frac{1}{a_R}\int_{B_R}1\cdot \chi_R(w)\dl(w)=1.
	%$$
	%We also obtain 
	\begin{align}
	|\widetilde{f}_R(\tau,x,y)|^2 &= \frac{1}{a_R^2}\left| \int_{B_R} \chi_R(y-w)f_R(\tau, x, w)\dl(w)\right|^2\\
	&\leq \frac{1}{a_R^2}\int_{B_R} |\chi_R(y-w)|^2\dl(w)\int_{B_R}|f_R(\tau, x, w)|^2\dl(w)\\
	&\leq \frac{\sigma_nR^n}{a_R^2}\int_{B_R}|f_R(\tau, x, w)|^2\dl(w)\label{eq:convolution}
	\end{align}
	for $(\tau, x, y)\in \Delta(a;r)\times (V+\ai \R^n)$, and
	\begin{align}
	\int_{\Delta(a;r)\times V}|\widetilde{f}_R(\tau,x,y)|^2e^{-\varphi(\tau, x)}\dl(\tau,x)&\leq \frac{\sigma_nR^n}{a_R^2} \int_{\Delta(a;r)\times (V+\ai B_R)}|f_R(\tau,x,w)|^2e^{-\varphi(\tau, x, w)}\dl(\tau,x,w)\\
	&\leq \frac{(\sigma_nR^n)^2}{a_R^2}\pi r^2 \int_V e^{-\varphi(a,x)}\dl(x)\\
	&\leq \left( \frac{R}{R-2\sqrt{R}}\right)^{2n} \pi r^2 \int_V e^{-\varphi(a,x)}\dl(x)\label{eq:tilda}
	\end{align}
	due to (\ref{eq:roptimal}).
	Here we use the fact that $\sigma_n(R-2\sqrt{R})^n\leq a_R$. 
	Note that $\{ R/(R-2\sqrt{R})\}_{R>\!>0}$ is decreasing and has an upper bound independent of $R$.
	For instance, if $R\geq 100$, we can estimate 
	\begin{equation}\label{eq:gutaitekinateisuu}
	\int_{\Delta(a;r)\times V}|\widetilde{f}_R(\tau,x,y)|^2e^{-\varphi(\tau, x)}\dl(\tau,x)\leq \left( \frac{5}{4}\right)^{2n}\pi r^2 \int_V e^{-\varphi(a,x)}\dl(x).
	\end{equation}
	We also obtain 
	$$
	\frac{\partial \widetilde{f}_R}{\partial y}(\tau,x,y)=\frac{1}{a_R}\int_{\R^n}\frac{\partial \chi_R}{\partial y}(y-w)f_R(\tau,x,w)\dl(w),
	$$
	and 
	\begin{align*}
	\left| \frac{\partial \widetilde{f}_R}{\partial y}(\tau,x,y) \right|^2 &\leq \frac{1}{a_R^2}\int_{B_R}\left| \frac{\partial \chi_R}{\partial y}(y-w)\right|^2\dl(w)\int_{B_R}|f_R(\tau,x,w)|^2\dl(w)\\
	&\leq \frac{\sigma_nR^n}{a_R^2}\left(\frac{C}{\sqrt{R}}\right)^2\int_{B_R}|f_R(\tau,x,w)|^2\dl(w).
	\end{align*}
	Repeating the argument above, we get 
	\begin{align}
	\int_{\Delta(a;r)\times V}\left| \frac{\partial \widetilde{f}_R}{\partial y}(\tau,x,y) \right|^2e^{-\varphi(\tau, x)}\dl(\tau,x)\leq \left( \frac{R}{R-2\sqrt{R}}\right)^{2n}\left(\frac{C}{\sqrt{R}}\right)^2(\pi r^2)\int_V e^{-\varphi(a,x)}\dl(x)\label{eq:bibun}
	\end{align}
	and $R/(R-2\sqrt{R})^{2n}\leq (5/4)^{2n}$ when $R\geq 100$ as well. 
	Since $\chi_R$ has compact support in $\{ |w|<R-\sqrt{R}\}$, for $y\in \{|y|<\sqrt{R}\}$, 
	we may assume that $|y-w|<R$ when we consider the integration
	$$
	\int_{\R^n}f_R(\tau,x,y-w)\chi_R(w)\dl(w)=\int_{B_{R-\sqrt{R}}}f_R(\tau,x,y-w)\chi_R(w)\dl(w).
	$$
	On $\{ |y|<\sqrt{R}/2\}$, we have 
	$$
	\frac{\partial \widetilde{f}_R}{\partial x}(\tau,x,y)=\frac{1}{a_R}\int_{B_{R-\sqrt{R}}}\frac{\partial f_R}{\partial x}(\tau,x,y-w)\chi_R(w)\dl(w), \frac{\partial \widetilde{f}_R}{\partial y}(\tau,x,y)=\frac{1}{a_R}\int_{B_{R-\sqrt{R}}}\frac{\partial f_R}{\partial y}(\tau,x,y-w)\chi_R(w)\dl(w)
	$$
	$$
	\frac{\partial \widetilde{f}_R}{\partial t}(\tau,x,y)=\frac{1}{a_R}\int_{B_{R-\sqrt{R}}}\frac{\partial f_R}{\partial t}(\tau,x,y-w)\chi_R(w)\dl(w), \frac{\partial \widetilde{f}_R}{\partial s}(\tau,x,y)=\frac{1}{a_R}\int_{B_{R-\sqrt{R}}}\frac{\partial f_R}{\partial s}(\tau,x,y-w)\chi_R(w)\dl(w),
	$$
	where $\tau =t+\ai s$. 
	Then we see that $\partial/\partial \bar{\tau}$ and $\partial / \partial \bar{z}$ commute with the integral as well, which implies that $\widetilde{f}_R$ is holomorphic on $\Delta(a;r)\times (V+\ai B_{\sqrt{R}/2})$. 
	
	(Step 3) Take the limit $R\to +\infty$. %Use the Arzel\`a-Ascoli theorem.
	
	%We set $B_n:=\{ y\in \R^n \mid |y|^2<n^2 \}\subset \R^n_y$ for $n\in \N$. 
	We fix a monotonically increasing sequence $\{ R_j\}_{j\in \N}$ of positive numbers such that $R_1$ is sufficiently large, $R_{j+1}>R_j$ and $\lim_j R_j=+\infty$. 
	%Note that for each $n\in \N$, there exists $R$
	We also take an exhaustion by compact sets $\{K_i\}_{i\in \N}$ in $\Delta(a; r)\times (V+\ai \R^n)$ such that $(K_1)^{\mathrm{o}}\neq \emptyset$, $K_i\subset (K_{i+1})^{\mathrm{o}}$ and $\cup_i K_i = \Delta(a; r)\times (V+\ai \R^n)$.
	For each $K_i$, we can obtain compact subsets $L_{i_1}\subset \Delta(a;r)$, $L_{i_2}\subset V$ and $L_{i_3}\subset \R^n$ satisfying $K_i\subset (L_{i_1}\times L_{i_2}\times L_{i_3})^\mathrm{o}$.
	
	First, we consider $L^2$-estimates on $K_1$. 
	It follows that there exists $R_{n_1}$ such that $L_{1_3}\subset \{ |y|<\sqrt{R_{n_1}}/2\}$, that is, for every $j\geq n_1$, $\widetilde{f}_j$ is holomorphic on $L_{1_1}\times L_{1_2}\times L_{1_3}$. 
	Note that thanks to (\ref{eq:tilda}) and (\ref{eq:gutaitekinateisuu}), we have 
	\begin{equation}
	\sup_{y\in L_{i_3}} \| \widetilde{f}_{R_j}(\cdot, \cdot, y)\|_{L^2_\varphi}\leq C<+\infty, \label{eq:upperbound}
	\end{equation}
	where $\| \widetilde{f}_{R_j}(\cdot, \cdot, y)\|_{L^2_\varphi}=\int_{\Delta(a;r)\times V} |\widetilde{f}_{R_j}(\tau, x, y)|^2e^{-\varphi(\tau, x)}\dl (\tau, x)$ and $C$ is a positive constant, which is independent of $R_j$, $K_i$ and $L_{i_1}, L_{i_2}, L_{i_3}$.
	For $i=1$, $\{ \sup_{y\in L_{1_3}} \| \widetilde{f}_{R_j}(\cdot, \cdot, y)\|_{L^2_\varphi}\}_j$ is a bounded sequence. 
	Hence, there exists a convergent subsequence $\{ \sup_{y\in L_{1_3}} \| \widetilde{f}_{R_{j_{1, k}}}(\cdot, \cdot, y)\|_{L^2_\varphi}\}_{j_{1, k}}$. 
	We may assume that $j_{1, 1}\geq n_1$. 
	Since $\varphi$ is locally bounded above, there is a positive constant $C_1$ such that $\varphi \leq C_1$, that is, $e^{-\varphi}\geq e^{-C_1}$ on $L_{1_1}\times L_{1_2}\times L_{1_3}$.
	Then we have that 
	\begin{align*}
	\sup_{y\in L_{1_3}}\| \widetilde{f}_{R_{j_{1, k}}}(\cdot, \cdot, y)-\widetilde{f}_{R_{j_{1, \ell}}}(\cdot, \cdot, y)\|^2_{L^2_\varphi} &\geq \sup_{y\in L_{1_3}}\int_{L_{1_1}\times L_{1_2}}|\widetilde{f}_{R_{j_{1, k}}}(\tau,x,y)-\widetilde{f}_{R_{j_{1, \ell}}}(\tau,x,y)|^2e^{-\varphi(\tau,x)}\dl(\tau,x)\\
	&\geq \frac{1}{|L_{1_3}|}\int_{L_{1_1}\times L_{1_2}\times L_{1_3}}|\widetilde{f}_{R_{j_{1, k}}}(\tau,x,y)-\widetilde{f}_{R_{j_{1, \ell}}}(\tau,x,y)|^2e^{-\varphi(\tau,x)}\dl(\tau,x, y)\\
	&\geq \frac{e^{-C_1}}{|L_{1_3}|}\int_{L_{1_1}\times L_{1_2}\times L_{1_3}}|\widetilde{f}_{R_{j_{1, k}}}(\tau,x,y)-\widetilde{f}_{R_{j_{1, \ell}}}(\tau,x,y)|^2\dl(\tau,x,y)\\
	&\geq C_{K_1,L_{1_1}, L_{1_2}, L_{1_3}} \sup_{(\tau,x,y )\in K_1}|\widetilde{f}_{R_{j_{1, k}}}(\tau,x,y)-\widetilde{f}_{R_{j_{1, \ell}}}(\tau,x,y)|^2
	\end{align*}
	for some positive constant $C_{K_1,L_{1_1}, L_{1_2}, L_{1_3}}>0$ since
	 $\widetilde{f}_{R_{j_{1, k}}}$ and $\widetilde{f}_{R_{j_{1, \ell}}}$ are holomorphic on $L_{1_1}\times L_{1_2}\times L_{1_3}$.
	 Then $\{ \widetilde{f}_{R_{j_{1, k}}} \}_k$ forms a Cauchy sequence in the space of continuous functions on $K_1$ with the sup norm. 
	 Hence, there exists a function $f_{K_1, \infty}$ on $K_1$ such that $\{ \widetilde{f}_{R_{j_{1, k}}} \}_k$ uniformly converges to $f_{K_1, \infty}$ on $K_1$.
	 Here $f_{K_1, \infty}$ is holomorphic in $K_1^\mathrm{o}$.
	 
	 Next, we consider the $L^2$-estimates on $K_2$. Repeating the above argument, we can get a convergent subsequence $\{ \widetilde{f}_{R_{j_{2, k}}} \}_k$ of $\{ \widetilde{f}_{R_{j_{1, k}}} \}_k$ and a function $f_{K_2, \infty}$. 
	 Since $\{ \widetilde{f}_{R_{j_{2, k}}} \}_k$ is also uniformly converging to $f_{K_2, \infty}$, it holds that $f_{K_2, \infty}|_{K_1}=f_{K_1, \infty}$. 
	 
	 By using the diagonal argument, we can finally conclude that there exists a holomorphic function $f_\infty$ on $\Delta(a;r)\times (V+\ai \R^n)$ such that  $\{ \widetilde{f}_{R_{j_{k, k}}} \}_k$ uniformly converges to $f_\infty$ on every compact set. 
	Then $\frac{\partial \widetilde{f}_{R_{j_{k, k}}}}{\partial y}$ also uniformly converges to $\frac{\partial f_\infty}{\partial y}$ on every compact set. 
	Fix any point $(\tau_0, x_0, y_0)\in \Delta(a;r)\times (V+\ai \R^n)$ and take $K_{n} \ni (\tau_0, x_0, y_0)$. 
	By (\ref{eq:bibun}), we have that 
	\begin{align*}
	e^{-C'}\int_{L_{{n}_1}\times L_{{n}_2}}\left| \frac{\partial \widetilde{f}_{R_{j_{k, k}}}}{\partial y}(\tau,x,y_0)\right|^2\dl(\tau,x)&\leq \int_{L_{{n}_1}\times L_{{n}_2}}\left| \frac{\partial \widetilde{f}_{R_{j_{k, k}}}}{\partial y}(\tau,x,y_0)\right|^2e^{-\varphi(\tau,x)}\dl(\tau,x)\\
	&\leq \left( \frac{R_{j_{k_0, k_0}}}{R_{j_{k_0, k_0}}-2\sqrt{R_{j_{k_0, k_0}}}}\right)^{2n}\left( \frac{C}{\sqrt{R_{j_{k_0, k_0}}}}\right)^2(\pi r^2)\int_V e^{-\varphi(a,x)}\dl(x)<+\infty
	\end{align*}
	for $k\geq k_0$. 
	Then we obtain 
	$$
	\int_{L_{{n}_1}\times L_{{n}_2}}\left| \frac{\partial {f}_\infty}{\partial y}(\tau,x,y_0)\right|^2\dl(\tau,x)\leq \frac{C''}{R_{j_{k_0, k_0}}}\left( \frac{R_{j_{k_0, k_0}}}{R_{j_{k_0, k_0}}-2\sqrt{R_{j_{k_0, k_0}}}}\right)^{2n}\int_V e^{-\varphi(a,x)}\dl(x)<+\infty
	$$
	for $C''>0$. Letting $k_0\to \infty$, we get 
	%$\frac{\partial f_\infty}{\partial y}(\tau,x,y)=0$ on an arbitrary compact set $K$, which implies that $f_\infty$ is a holomorphic function independent of $y$. 
	$\frac{\partial f_\infty}{\partial y}(\tau,x,y_0)=0$ on $L_{{n}_1}\times L_{{n}_2}$, that is, $\frac{\partial f_\infty}{\partial y}(\tau_0,x_0,y_0)=0$. 
	Since $(\tau_0,x_0,y_0)$ is arbitrary, $f_\infty$ is a holomorphic function independent of $y$. 
	Hence, $f_\infty$ is independent of $z=x+\ai y$. 
	Then we define the well-defined holomorphic map $f:\Delta(a;r)\to \C$ by $f(\tau):=f_\infty(\tau,x,y)$.
	%Fix $j_0\in \N$. 
	For $j\in \N$, by (\ref{eq:tilda}), we get 
	\begin{align*}
	\int_{\Delta(a;r)\times V}|\widetilde{f}_{R_{j_{k, k}}}(\tau,x,y)|^2e^{-\varphi(\tau,x)}\dl(\tau,x)&\leq \left( \frac{R_{j_{k, k}}}{R_{j_{k, k}}-2\sqrt{R_{j_{k, k}}}}\right)^{2n}\pi r^2 \int_V e^{-\varphi(a,x)}\dl(x).
	%&\leq \left( \frac{R_{j_0}}{R_{j_0}-2\sqrt{R_{j_0}}}\right)^{2n}\pi r^2 \int_V e^{-\varphi(a,x)}\dl(x).
	\end{align*}
	%Since $\| F_{R_j}(y)-F_\infty(y)\|^2_{L^2_\varphi}\to 0$ as $j\to \infty$ for $y\in \R^n$, 
	Taking the limit $k \to \infty$, thanks to Fatou's lemma, 
	we have that 
	$$
	\int_{\Delta(a;r)\times V}|{f}_{\infty}(\tau,x,y)|^2e^{-\varphi(\tau,x)}\dl(\tau,x)\leq \pi r^2 \int_V e^{-\varphi(a,x)}\dl(x),
	$$
	%Taking the limit $j_0 \to \infty$, we get 
	that is,
	$$
	\int_{\Delta(a;r)\times V}|f(\tau)|^2e^{-\varphi(\tau,x)}\dl(\tau,x)\leq \pi r^2 \int_V e^{-\varphi(a,x)}\dl(x).
	$$
	
	We also have that 
	\begin{align*}
	\widetilde{f}_{R_{j_{k, k}}}(a,x,0)&=\frac{1}{a_{R_{j_{k, k}}}}\int_{\R^n}f_{R_{j_{k, k}}}(a,x,-w)\cdot \chi_{R_{j_{k, k}}}(w)\dl(w)\\
	&=\frac{1}{a_{R_{j_{k, k}}}}\int_{B_{R_{j_{k, k}}}}1\cdot \chi_{R_{j_{k, k}}}(w)\dl(w)\\
	&=1.
	\end{align*}
	%on $\{ |y|<\sqrt{R_j} \}$.
	Then we see that $f(a)=f_\infty(a,x,0)=\lim_{k\to \infty}\widetilde{f}_{R_{j_{k, k}}}(a,x,0)=1$, which completes the proof. 
\end{proof}

\begin{remark}
	The constant in the $L^2$-extension of $\widetilde{f}_R$ is not optimal and changes for each $R>0$ (see (\ref{eq:tilda})). However, by taking the limit ${R\to \infty}$, we can estimate the $L^2$-norm of $f_\infty$ with the optimal constant. 
\end{remark}

\section{A simple proof of Pr\'ekopa's theorem}\label{sec:simpleprekopa}

In this section, applying Theorem \ref{thm:optimaltube}, we give a simple proof of Pr\'ekopa's theorem. 
The proof based on a non-optimal $L^2$-extension theorem for the ``complex" version of Pr\'ekopa's theorem also appeared in \cite{DWZZ18}, \cite{DWZZ19}.
Our main purposes are to establish the optimal $L^2$-extension theorem on tube domains and to give a proof of Pr\'ekopa's theorem in the ``real" setting directly.
%By establishing the optimal $L^2$-extension theorem with tube domains above, we can give a proof of Pr\'ekopa's theorem in the ``real" setting directly, which is a main purpose of this article. 

First, we consider the following case. 

\begin{theorem}
	Let $V$ be a convex domain in $\R^n$ and $\varphi$ be a convex function on $\R_t\times V_x$. Assume that $V$ is bounded and 
	$$
	e^{-\Phi(t)}:=\int_{V}e^{-\varphi(t,x)}\dl(x)<+\infty
	$$
	for each $t\in R$. Then $\Phi$ is convex.
\end{theorem}

\begin{proof}
	We consider the following tube domains $\R_t+\ai \R_s$ and $V_x+\ai \R^n_y$, and set $\tau=t+\ai s, z=x+\ai y$. 
	We also let $\widehat{\Phi}(\tau):=\Phi(t)$ be a function on $(\R+\ai \R)_\tau$ and $\widehat{\varphi}(\tau,z):=\varphi(t,x)$ be a function on $(\R+\ai \R)_\tau \times (V+\ai \R^n)_z$. 
	Then it clearly holds that $\widehat{\varphi}$ is a plurisubharmonic function and    
	$$
	e^{-\widehat{\Phi}(\tau)}=\int_V e^{-\widehat{\varphi}(\tau,z)}\dl(x). 
	$$
	It is enough to show that $\widehat{\Phi}$ is plurisubharmonic. 
	Note that $\widehat{\Phi}$ is independent of $s$ and $\widehat{\varphi}$ is independent of $s$ and $y$. 
	We only need to show that $\widehat{\Phi}$ satisfies the minimal extension property since $\widehat{\Phi}$ is upper semi-continuous thanks to Fatou's lemma (cf. Theorem \ref{thm:mep}). 
	
	Take a point $a\in \R+\ai \R$ and $r>0$. 
	Then, by Theorem \ref{thm:optimaltube}, there exists a holomorphic function $f$ on $\Delta(a;r)$ satisfying $f(a)=1$ and 
	$$
	\int_{\Delta(a;r)\times V} |f(\tau)|^2 e^{-\widehat{\varphi}(\tau,x)}\dl(\tau,x)\leq \pi r^2 \int_V e^{-\widehat{\varphi}(a,x)}\dl(x)<+\infty,
	$$
	that is, 
	$$
	\frac{1}{\pi r^2}\int_{\Delta(a;r)}|f(\tau)|^2e^{-\widehat{\Phi}(\tau)}\dl(\tau)\leq e^{-\widehat{\Phi}(a)},
	$$
	which completes the proof. 
	%Since $\widehat{\Phi}$ satisfies the minimal extension property, we can conclude that $\widehat{\Phi}$ is plurisubharmonic (cf. Theorem \ref{thm:mep})
	
\end{proof}

%\begin{remark}
	%In the case that the dimension of the base domain is greater than $1$, we can also prove Pr\'ekopa's theorem in the same way since Theorem \ref{thm:mep} still holds in this case. 
%\end{remark}

\begin{remark}
	The above type proof can be applied to the complex Pr\'ekopa theorem as well. 
\end{remark}

If $V$ is an unbounded convex domain such as $\R^n$, we need to take a convex exhaustion. 
We only show the proof in the case that $V=\R^n$ without loss of generality.

\begin{theorem}
	Keep the notation above. Set $V=\R^n$. 
	Suppose that 
	$$
	e^{-{\Phi}(t)}:=\int_{\R^n}e^{-\varphi(t,x)}\dl(x)<+\infty. 
	$$
	Then $\Phi$ is convex. 
\end{theorem} 

\begin{proof}
	Let $B_j:=\{|x|<j\}\subset \R^n$ for $j\in \N$. We define 
	$$
	e^{-\Phi_j(t)}:=\int_{B_j}e^{-\varphi(t,x)}\dl(x) <+\infty. 
	$$
	Then we know that $\Phi_j$ is convex. 
	It holds that $\Phi_j$ is decreasing to $\Phi$. 
	Then $\Phi$ is convex as well. 
\end{proof}

%%%%%%%%%%%%%%%%%%%%%%%%%%%%%%%%%%%%%%%%%%%%%%%%%%%%%%
%%%%%%%%%%%%%%%%%%%%%%%%%%%%%%%%%%%%%%%%%%%%%%%%%%%%%%

%%%%%%%%%%%% References %%%%%%%%%%%%%
%%
%<Author name> is written as Initial of Given Name, and Family Name.
%<Title> is written in roman letters.
%<Journal name> should be abbreviated according to
% the MR Serials Abbreviations List of Mathematical Reviews:
% (Abbreviations of Names of Serials; http://www.ams.org/mr-database)
%For <Pages>, use en-dash "--" between page numbers.
%%


\begin{thebibliography}{99}
%\bibitem[AG62]{AG62}	A. Andreotti and H. Grauert, \emph{Th\'eor\`eme de finitude pour la cohomologie des espaces complexes}, Bull. Soc. Math. France \textbf{90}, (1962), 193-259.
 %\bibitem[AV65]{AV65} A.~Andreotti and E.~Vesentini, \emph{Carleman estimates for the Laplace-Beltrami equation in complex manifolds}, Publ. Math. I.H.E.S. \textbf{25}, (1965), 81-130.
 \bibitem[Ber98]{Ber98} B.~Berndtsson, \emph{Prekopa's theorem and Kiselman's minimum principle for plurisubharmonic functions}, Math. Ann. \textbf{312}, (1998), 785-792. 
 %\bibitem[Ber09]{Ber09} B.~Berndtsson, \emph{Curvature of vector bundles associated to holomorphic fibrations}, Ann. of Math. \textbf{169}, (2009), no.~2, 531-560. 
 %\bibitem[Ber10]{Ber10} B.~Berndtsson, \emph{An introduction to things $\dbar$, Analytic and algebraic geometry},IAS/Park City Math. Ser., vol. 17, Amer. Math. Soc., Providence, RI, 2010, pp. 7-76.
 %\bibitem[BP08]{BP08} B.~Berndtsson and M.~P{\u{a}}un, \emph{Bergman kernels and the pseudoeffectivity of relative canonical bundles}, Duke Math. J. \textbf{145}, (2008), no.~2, 341--378.
 %\bibitem[BP10]{BP} B.~Berndtsson and M.~P\u{a}un, \emph{Bergman kernels and subadjunction}, arXiv:1002.4145.
 \bibitem[Blo13]{Blo13} Z.~B\l ocki, \emph{Suita conjecture and the Ohsawa-Takegoshi extension theorem}, Invent. Math. \textbf{193}, (2013), 149-158. 
 \bibitem[BL76]{BL76} H.~J.~Brascamp and E.~H.~Lieb, \emph{On extensions of the Brunn-Minkowski and Pr\'ekopa-Leindler theorems, including inequalities for log concave functions, and with an application to the diffusion equation}, J. Funct. Anal. \textbf{22}, (1976), no.~4, 366-389. 
 %\bibitem[CF90]{CF90} F. Campana and H. Flenner, \emph{A characterization of ample vector bundles on a curve}, Math. Ann. \textbf{287}, (1990), no.~4, 571-575.
 %\bibitem[CP17]{CP17} J. Cao and M.~P\u{a}un, \emph{Kodaira dimension of algebraic fiber spaces over abelian varieties}, Invent. Math. \textbf{207}, (2017), 345-387.
%\bibitem[Cor05]{Cor05} D. Cordero-Erausquin, \emph{On Berndtsson’s generalization of Pr\'ekopa’s theorem}, Math. Z. \textbf{249}, (2005), 401-410.
 %\bibitem[deC98]{deC98} M.~A.~A. de~Cataldo, \emph{Singular {H}ermitian metrics on vector bundles}, J. Reine Angew. Math. \textbf{502}, (1998), 93--122.
 %\bibitem[Dem82]{Dem82} J.-P. Demailly, \emph{Estimations $L^2$ pour l'op\'erateur $\dbar$ d'un fibr\'e vectoriel holomorphe semi-positif au dessus d'une vari\'et\'e K\"ahl\'erienne compl\`ete}, Ann. Sci. Ec. Norm. Sup. \textbf{15}, (1982), 457-511.
 %\bibitem[Dem93]{Dem93} J.-P. Demailly, \emph{A numerical criterion for very ample line bundles}, J. Differential Geom. \textbf{37}, (1993), 323-374. 
 %\bibitem[Dem]{Dem12} J.-P. Demailly, \emph{Analytic methods in algebraic geometry}, Surveys of Modern Mathematics, vol.~1, International Press, Somerville, MA; Higher Education Press, Beijing, 2012.
 %\bibitem[Dem-book]{DemCom} J.-P. Demailly, \emph{Complex analytic and differential geometry}, http://www-fourier.ujf-grenoble.fr/~demailly/manuscripts/agbook.pdf.
%\bibitem[DS]{DS} J.-P. Demailly and H.~Skoda, \emph{Relations entre les notions de positivit\'e de P.~A.~Griffiths et de S.~Nakano}, S\'eminaire P.~Lelong-H.~Skoda (Analyse), ann\'ee 1978/79, Lecture notes in Math. no \textbf{822}, Springer-Verlag, Berlin (1980) 304-309.
 \bibitem[DNW19]{DNW19} F.~Deng, J.~Ning, and Z.~Wang, \emph{Characterizations of plurisubharmonic functions}, arXiv:1910.06518.
 \bibitem[DNWZ20]{DNWZ20} F.~Deng, J.~Ning, Z.~Wang, and X.~Zhou, \emph{Positivity of holomorphic vector bundles in terms of $L^p$-conditions of $\bar{\partial}$}, arXiv:2001.01762. 
 \bibitem[DWZZ18]{DWZZ18} F.~Deng, Z.~Wang, L.~Zhang, and X.~Zhou, \emph{New characterizations of plurisubharmonic functions and positivity of direct image sheaves}, arXiv:1809.10371.
 \bibitem[DWZZ19]{DWZZ19} F.~Deng, Z.~Wang, L.~Zhang, and X.~Zhou, \emph{A new proof of Kiselman's minimum principle for plurisubharmonic functions}, Comptes Rendus Math\'ematique \textbf{357}, (2019), no.~4, 345-348. 
 %\bibitem[DWZZ20]{DWZZ20} F.~Deng, Z.~Wang, L.~Zhang, and X.~Zhou, \emph{Linear invariants of complex manifolds and their plurisubharmonic variations}, J. Funct. Anal. \textbf{279}, (2020), https://doi.org/10.1016/j.jfa.2020.108514. 
 %\bibitem[DZZ14]{DZZ14} F.~Deng, H.~Zhang, and X.~Zhou, \emph{Positivity of direct images of positively curved volume forms}, Math. Z. \textbf{278}, (2014), no.~1-2, 347-362. 
 %\bibitem[Gri69]{Gri69}P. A. Griffiths, \emph{Hermitian differential geometry, Chern classes, and positive vector bundles}, Global Analysis, papers in honor of K. Kodaira, Princeton Univ. Press, Princeton, (1969), 181-251.
 \bibitem[GZ15]{GZ15} Q.~Guan and X.~Zhou, \emph{A solution of an $L^2$ extension problem with an optimal estimate and applications}, Ann. of Math. (2), \textbf{181}, (2015), no.~3, 1139-1208.
 \bibitem[HPS18]{HPS18} C.~Hacon, M.~Popa, and C.~Schnell, \emph{Algebraic fiber spaces over Abelian varieties: around a recent theorem by Cao and Paun}, Contemp. Math. \textbf{712} (2018), 143-195. 
 %\bibitem[H\"or65]{Hor65} L.~H\"ormander, \emph{$L^2$ estimates and existence theorems for the $\dbar$ operator}, Acta Math. \textbf{113}, (1965), 89-152. 
 %\bibitem[HI20]{HI20} G.~Hosono and T.~Inayama, \emph{A converse of H\"ormander's $L^2$-estimate and new positivity notions for vector bundles}, Sci. China Math. (2020). https://doi.org/10.1007/s11425-019-1654-9.
 %\bibitem[Ina20]{Ina18} T.~Inayama, \emph{$L^2$ estimates and vanishing theorems for holomorphic vector bundles equipped with singular Hermitian metrics}, Michigan Math. J. \textbf{69}, (2020), 79-96.
 %\bibitem[Ina20]{Ina20} T.~Inayama, \emph{Nakano positivity of singular Hermitian metrics and vanishing theorems of Demailly-Nadel-Nakano type}, arXiv:2004.05798. 
 \bibitem[Kis78]{Kis78} C.~Kiselman, \emph{The partial Legendre transformation for plurisubharmonic functions}, Invent. Math. \textbf{49}, (1978), no.~2, 137-148. 
 %\bibitem{LSY13} K.-F.~Liu, X.-F. Sun, and X.-K. Yang, \emph{Positivity and vanishing theorems for ample vector bundles}, J. Algebraic Geom. \textbf{22}, (2013), 303-331.   
 %\bibitem[LY14]{LY14} K. Liu and X. Yang, \emph{Curvatures of direct image sheaves of vector bundles and applications}, J. Differential Geom. \textbf{98}, (2014), 117-145.
 %\bibitem[MT08]{MT08} C. Mourougane and S. Takayama, \emph{Hodge metrics and the curvature of higher direct images}, Ann. Sci. \'Ec. Norm. Sup\'er. \textbf{41}, (2008), 905-924.
 %\bibitem[MT09]{MT09} C. Mourougane and S. Takayama, \emph{Extension of twisted Hodge metrics for K\"ahler morphisms}, J. Differential Geom. \textbf{83}, (2009) 131-161.
 %\bibitem[Nad90]{Nad90} A.~M.~Nadel, \emph{Multiplier ideal sheaves and K\"ahler-Einstein metrics of positive scalar curvature}, Ann. of Math. (2) \textbf{132}, (1990), no.~3, 549-596.
 %\bibitem[NZZ16]{NZZ} J.~Ning, H.~Zhang, and X.~Zhou, \emph{On $p$-Bergman kernel for bounded domains in $\mathbb{C}^n$}, Comm. Anal. Geom. \textbf{24}, (2016), no.~4, 887-900.
 \bibitem[OT87]{OT87} T.~Ohsawa and K.~Takegoshi, \emph{On the extension of $L^2$ holomorphic functions}, Math. Z. \textbf{195}, (1987), no.~2, 197-204. 
 %\bibitem[PT18]{PT18} M.~P{\u{a}}un and S.~Takayama, \emph{Positivity of twisted relative pluricanonical bundles and their direct images}, J. Algebraic Geom. \textbf{27}, (2018), 211-272. 
 \bibitem[Pr\'e73]{Pre73} A.~Pr\'ekopa, \emph{On logarithmic concave measures and functions}, Acta. Sci. Math. (Szeged) \textbf{34}, (1973), 335-343. 
 %\bibitem[Rau15]{Rau15} H.~{Raufi}, \emph{Singular hermitian metrics on holomorphic vector bundles}, Ark. Mat. \textbf{53}, (2015), no.~2, 359-382.
%\bibitem[Siu76]{Siu76} Y.~T.~Siu, \emph{Every Stein subvariety admits a Stein neighborhood}, Invent. Math. \textbf{38}, (1976/77), 89-100.
%
%\bibitem[Sko72]{Sko72}H.~Skoda, \emph{Sous-ensembles analytiques d'ordre fini ou infini dans $\mathbb{C}^n$}, Bull. Soc. Math. France, \textbf{100} (1972), 353-408.
%\bibitem[Ume73]{Ume73} H. Umemura, \emph{Some results in the theory of vector bundles}, Nagoya Math. J. \textbf{52}, (1973), 97-128.
%\bibitem[Yan18]{Yan18} X. Yang, \emph{RC-positivity, rational connectedness and Yau’s conjecture}, Camb. J. Math. \textbf{6}, (2018), no.~2, 183-212. 
%\bibitem[Yan19]{Yan19} X. Yang, \emph{A partial converse to the Andreotti-Grauert theorem}, Compos. Math. \textbf{155}, (2019), 89-99.
\end{thebibliography}
\end{document}